\documentclass[12pt]{iopart}

\usepackage{amsthm,amssymb,graphicx,pst-all,color,enumerate}
\usepackage[latin1]{inputenc}
\usepackage{iopams}
\usepackage{xspace} 
\usepackage{algorithm}
\usepackage{algorithmic} 
\usepackage{tikz}
\newcommand\copyrighttext{%
  \footnotesize
  This is an author-created, un-copyedited version of an article published in Inverse Problems. IOP Publishing Ltd is not responsible for any errors or omissions in this version of the manuscript or any version derived from it. The Version of Record is available online at http://dx.doi.org/10.1088/0266-5611/32/3/035001.}
\newcommand\copyrightnotice{%
\begin{tikzpicture}[remember picture,overlay]
\node[anchor=north] at (current page.north)
{\fbox{\parbox{\dimexpr\textwidth-\fboxsep-\fboxrule\relax}{\copyrighttext}}};
\end{tikzpicture}%
} 

\newtheorem{Theorem}{Theorem}[section]		
\newtheorem{Corollary}[Theorem]{Corollary}
\newtheorem{Proposition}[Theorem]{Proposition}
\newtheorem{Lemma}[Theorem]{Lemma}
\newtheorem{Definition}[Theorem]{Definition}
\newtheorem{Remark}[Theorem]{Remark}
\newtheorem{Procedure}[Theorem]{Procedure}

\def\req#1{{\rm(\ref{#1})}}
\def\norm#1{\hspace{0.2ex} \|#1\| \hspace{0.2ex}}
\renewcommand{\sp}[2]{\langle#1,#2\rangle}
\newcommand{\transp}{^{^{\scriptstyle\intercal}}}

\newcommand{\R}{\mathbb R}

\newcommand{\N}{\mathbb N}
\newcommand{\ei}{\mathrm{ess\,inf\,}}
\newcommand{\esup}{\mathrm{ess\,sup\,}}
\newcommand{\mcP}{\mathcal{P}}
\newcommand{\mcD}{\mathcal{D}}
\newcommand{\mcF}{\mathcal{F}}
\newcommand{\mcB}{\mathcal{B}}
\newcommand{\Rpp}{\R_+^p}
\newcommand{\uc}{\underline{c}}

\newcommand{\FD}{Fr\'{e}chet derivative\xspace}

\graphicspath{{plots/}} 

\begin{document}

\copyrightnotice 

\title[A Reduced Basis Landweber method for nonlinear inverse problems]{A Reduced Basis Landweber method for nonlinear inverse problems}
\author{Dominik Garmatter$^1$, 
Bernard Haasdonk$^2$, Bastian Harrach$^1$}
\address{$^1$ Department of Mathematics,
Goethe University Frankfurt, Germany}
\address{$^2$ Department of Mathematics,
University of Stuttgart, Germany}
\eads{\mailto{garmatter@math.uni-frankfurt.de}, \mailto{haasdonk@math.uni-stuttgart.de}, \mailto{harrach@math.uni-frankfurt.de}}
\begin{abstract}
We consider parameter identification problems in parametrized partial differential equations (PDE). This leads to nonlinear ill-posed inverse problems. One way to solve them are iterative regularization methods, which typically require numerous amounts of forward solutions during the solution process. In this article we consider the nonlinear Landweber method and want to couple it with the reduced basis method as a model order reduction technique in order to reduce the overall computational time. In particular, we consider PDEs with a high-dimensional parameter space, which are known to pose difficulties in the context of reduced basis methods. 
We present a new method that is able to handle such high-dimensional parameter spaces by combining the nonlinear Landweber method with adaptive online reduced basis updates.
It is then applied to the inverse problem of reconstructing the conductivity in the stationary heat equation.
\end{abstract}



\section{Introduction}
The numerical solution of nonlinear inverse problems such as the identification of a parameter in a partial differential equation (PDE) from a noisy solution of the PDE via iterative regularization methods, e.g.\ the Land\-weber method or Newton-type methods, see, e.g., \cite{kaltenbacher2008iterative, engl1996regularization, rieder2003keine} for a detailed overview, usually requires numerous amounts of forward solutions of the respective PDE. Since this can be very time-consuming, it is highly desirable to speed up the solution process. 

The reduced basis method, see, e.g., \cite{RB_master_paper, Ha14RBTut} for a general survey, is a model order reduction technique that can yield a significant decrease in the computational time of the PDE solution, especially in a many-query context. The classical reduced basis framework aims at constructing a global reduced basis space that is a low-dimensional subspace of the solution space of the PDE providing accurate approximations to the PDE-solution for every parameter in the parameter domain. 
A possible way to construct such a space is to select \emph{meaningful} parameters and choose the corresponding PDE-solutions, the so called \emph{snapshots}, as basis vectors of the reduced basis space. Via Galerkin projection of the problem onto the reduced basis space, the reduced basis approximation can be computed. An offline/online decomposition of the procedure allows for the efficient and rapid computation of the reduced basis approximation for many different parameters.
Using a global reduced basis space, one can replace the expensive forward solution or a functional evaluation of it with the corresponding reduced basis quantity in the  solution procedure of a given inverse problem. We call this the direct approach and it has successfully been applied to various problems, see, e.g., \cite{nguyenrozza2009reduced, nguyenPHDThesis, HoangDentalTissuereduced}, but we want to stress that in those references the parameter space was bounded and of low dimension, which is required for the construction of a global reduced basis space.

The contribution of this paper to the field is the application of reduced basis methods to a parameter identification problem with a very high-dimensional and unbounded parameter space.
To this end we want to study the inverse model problem: given a solution $u(x),\, x\in\Omega$, of
\begin{eqnarray}
\label{eq:introduction_PDE_RBLpaper}
\nabla\cdot(\sigma(x)\nabla u(x)) &= 1,\quad x\in\Omega,\quad\mathrm{and}\quad u(x) &= 0,\quad x\in \partial\Omega,
\end{eqnarray}
identify the parameter $\sigma(x)$, with $\Omega\subset\R^2$ a bounded domain.
This is an example of recovering an image of the thermal conductivity in the stationary heat equation with constant heat source.
 Typically, instead of $u$ a noisy measurement $u^\delta$ is known, with $\norm{u-u^\delta}\leq \delta$ and noise level $\delta >0$. Since this problem is ill-posed, regularization techniques have to be applied. 
 We choose the nonlinear Landweber method in this article. 

The aim of this paper is the development of a new method to solve nonlinear inverse problems with high-dimensional parameter spaces, where our approach is based on  the ideas developed by Druskin and Zaslavsky \cite{Druskin_Zaslavski}.
We will combine the nonlinear Landweber method with the main ideas of the reduced basis method: instead of constructing a global reduced basis space, providing accurate approximations for every parameter in the parameter domain, as it is usually the case in reduced basis methods, see, e.g., \cite{DH14b}, we will adaptively construct a small problem-specific reduced basis space that may only be useful for the reconstruction of a single conductivity. 
 This will break the typical offline/online framework of reduced basis methods. A critical question then will be the  selection of the snapshots for this problem-oriented space. 
We will develop termination criteria that, together with the nonlinear Landweber method projected onto the current reduced basis space, will not only select meaningful parameters for space enrichment but also serve the solution of the posed inverse problem.
 Therefore, we adaptively enrich our reduced basis space to fit the region of the parameter space that is required to reconstruct the desired conductivity, while also reconstructing it. This will allow for the numerical treatment of very high-dimensional parameter spaces. We note that the nonlinear Landweber method being a regularization method has been extensively studied and analyzed, see, e.g., \cite{Hanke_Neubauer_Scherzer_Land_conv, engl1996regularization, kaltenbacher2008iterative, Hanke_LW_13_tau1}.

Ideas similar to our adaptive approach have been applied to a parameter estimation problem arising from the modeling of lithium ion batteries \cite{lass2014PHDThesis}, a subsonic aerodynamic shape optimization problem \cite{zahr2014progressive}, both leading to an optimization problem constrained with a nonlinear PDE, and a Bayesian inversion approach, where the parameter is modelled as a random variable, using Markov chain Monte Carlo (MCMC) methods \cite{cuiWilcox2014datadriveninversion}.

The remainder of this paper is organized as follows. In Section \ref{sec:Problem formulation_RBLpaper} we will present a mathematical formulation of the model problem, for which the nonlinear Landweber method is known to converge locally.
Section \ref{sec:chapter3_RBLpaper} contains a brief discretization as well as the key ingredients of the classical reduced basis method. 
In Section \ref{sec:RBL_RBLpaper} we will develop the new method, comment on its implementation and present numerical results. Final conclusions are drawn in Section \ref{sec:Conclusion_RBLpaper}. 
 

\section{Problem formulation}
\label{sec:Problem formulation_RBLpaper}

We present a well-known setting in which the Landweber method applied to \req{eq:introduction_PDE_RBLpaper} does converge locally.
Throughout this section we assume $\Omega\subset\R^2$ to be a bounded domain with $C^2$-boundary. 
Following \cite{Ito_Kunisch_aex, Hanke_97_a_example_2d, Kaltenbacher_Banachspaces_a_example_2d}, we choose the parameter space
\begin{eqnarray*}
H_+^{2}(\Omega) := \{\sigma(x)\in H^2(\Omega) \mid \ei\sigma(x)>0\},
\end{eqnarray*}
with $H^2(\Omega)$ the usual Sobolev space. Since $\Omega$ has a $C^2$-boundary, $H^2(\Omega)$ embeds continuously into $L^\infty(\Omega)$, cf. \cite[Theorem $4.12$]{adams2003sobolev}, such that taking the essential infimum of a function in $H^2(\Omega)$ is a continuous mapping and $H_+^2(\Omega)$ is an open subset of $H^2(\Omega)$. We consider the PDE: for given $\sigma(x)\in H_+^2(\Omega)$ find the (weak) solution $u\in H_0^1(\Omega)$ of
\begin{eqnarray}
\label{eq:PDE_RBLpaper}
\nabla\cdot(\sigma(x)\nabla u(x)) = 1.
\end{eqnarray}
The corresponding parameter-to-solution map is defined as
\begin{eqnarray}
\label{eq:forward_operator_RBLpaper}
\eqalign{
\mcF :\mcD(\mcF) := H_+^2(\Omega)\subset H^2(\Omega) \longrightarrow L^2(\Omega)\\
\mathcal{F}(\sigma) = u \quad u\in H_0^1(\Omega)\subset L^2(\Omega)\mbox{ solving}\\
b(u,w;\sigma) = f(w) \quad\mbox{ for all } w\in H_0^1(\Omega),\\
b(u,w;\sigma) := \int_\Omega \sigma \nabla u \cdot \nabla w\,dx , \quad f(w) := -\int_\Omega w\,dx.}
\end{eqnarray}
The associated \emph{inverse problem} is
\begin{eqnarray}
\label{eq:inverse problem_RBLpaper}
\mbox{for }u\in L^2(\Omega)\mbox{ find }\sigma\in H_+^2(\Omega)\mbox{ such that }\mathcal{F}(\sigma) = u.
\end{eqnarray}
Typically, instead of $u$ a noisy measurement $u^\delta\in L^2(\Omega)$ is known, with $\norm{u^\delta - u}_{L^2(\Omega)}\leq\delta$ and noise level $\delta >0$, such that regularization techniques have to be applied since simple inversion fails due to the ill-posedness of the problem. Throughout this paper we assume the knowledge of $\delta$.
\begin{Remark}\hfill
\label{remark:PDE_RBLpaper}
\begin{enumerate}[(i)]
\item For $\sigma\in H_+^2(\Omega)$ we have $\vert b(u,u;\sigma)\vert \geq \alpha(\sigma)\norm{u}_{L^2(\Omega)}^2$, with $\alpha(\sigma) := \frac{\ei \sigma(x)}{C_{PF}^2} > 0$, where $C_{PF}$ is the Poincar\'{e}-Friedrich constant of $\Omega$. Therefore, the bilinear form $b(\cdot , \cdot\, ;\sigma)$ is coercive for all $\sigma\in H_+^2(\Omega)$. 
Since $H^2(\Omega)$ embeds continuously into $L^\infty(\Omega)$, $b$ is also continuous for every $\sigma\in H_+^2(\Omega)$ with continuity constant $\gamma(\sigma) := \esup \sigma(x)<\infty$. The linear form $f$ is continuous as well such that the Lemma of Lax-Milgram guarantees existence and uniqueness of a solution of \req{eq:forward_operator_RBLpaper} in $H_0^1(\Omega)$.
\item $L^2(\Omega)$ is chosen as the image space of $\mcF$ because we consider it to be more realistic for a measurement $u^\delta\in L^2(\Omega)$ to be close to the exact data $u\in H_0^1(\Omega)$ in the $L^2$-norm rather than in the $H^1$-norm.
\item We acknowledge that the inverse problem stated in \req{eq:inverse problem_RBLpaper} is actually a linear problem since the data $u\in L^2(\Omega)$ is known over the whole domain $\Omega$. The problem becomes truly nonlinear if the data is only known on a proper subdomain $\tilde{\Omega}\subset\Omega$. We formulate this \emph{partial inverse problem}
\begin{eqnarray}
\label{eq:partial_inverse problem_RBLpaper}
\mbox{for }u\in L^2(\tilde{\Omega})\mbox{ find }\sigma\in H_+^2(\Omega)\mbox{ such that }\tilde{\mcF}(\sigma) = u,
\end{eqnarray}
with $\tilde{\mcF}:=E \circ\mcF$ and a restriction operator $E:L^2(\Omega)\longrightarrow L^2(\tilde{\Omega})$. Since the scope of this work is the connection of iterative regularization methods and the reduced basis method, we will continue to consider \req{eq:inverse problem_RBLpaper} for reasons of simplicity. Still, the resulting method derived in Section \ref{sec:RBL_RBLpaper} is applicable to the partial problem \req{eq:partial_inverse problem_RBLpaper} and we will provide corresponding  numerical results in Section \ref{subsec:RBL_Exp_RBLpaper}.
\item We mention that in the chosen setting (but also in general) the inverse problem \req{eq:inverse problem_RBLpaper} (and more so \req{eq:partial_inverse problem_RBLpaper}) is not uniquely solvable. In \cite{Ito_Kunisch_aex} Ito and Kunisch provide an overview of existing results on this topic and show the injectivity of $\mcF$ with respect to a reference parameter under certain assumptions. A recent result on the uniqueness of \req{eq:inverse problem_RBLpaper} with $C^2$-parameter is given by Knowles \cite{knowles1999uniqueness}, where his techniques are based on the work of Richter \cite{richter1981uniqueness}.
\end{enumerate}
\end{Remark}

If we consider $\mcD(\mcF) = L_+^\infty(\Omega)\subset L^2(\Omega)$ as definition space of $\mcF$, it is a well-known result that for each $\sigma\in L_+^\infty(\Omega)$ and each direction $\kappa\in L^2(\Omega)$ with $\sigma + \kappa \in L_+^\infty(\Omega)$ (note that $L_+^\infty(\Omega)$ is not an open subset of $L^2(\Omega)$)
\begin{eqnarray}
\label{eq:fd_property_Linfty_RBLpaper}
\lim_{\norm{\kappa}_\infty \rightarrow 0}\frac{\norm{\mcF(\sigma + \kappa) - \mcF(\sigma) - \mcF'(\sigma)\kappa}_{L^2(\Omega)}}{\norm{\kappa}_\infty} = 0
\end{eqnarray}
holds, with a linear and continuous operator $\mcF'(\sigma)$ given by
\begin{eqnarray}
\label{eq_frechetderivative_RBLpaper}
\eqalign{
\mcF'(\sigma)(\cdot):L^2(\Omega) \longrightarrow L^2(\Omega)\\
\mcF'(\sigma)\kappa = v\quad v\in H_0^1(\Omega)\subset L^2(\Omega) \mbox{ solving}\\
b(v,w;\sigma) = g(w;\kappa)\quad\mbox{ for all }w\in H_0^1(\Omega),\\
b(v,w;\sigma) := \int_\Omega \sigma \nabla v \cdot \nabla w\,dx , \quad g(w;\kappa) := -\int_\Omega \kappa \nabla u^\sigma \cdot \nabla w \,dx,
}
\end{eqnarray}
where $u^\sigma$ abbreviates $\mcF(\sigma)$. The Appendix of this article contains a proof of the above statement. 
Since $H^2(\Omega)$ embeds continuously into $L^\infty(\Omega)$ and $H_+^2(\Omega)$ is an open subset of $H^2(\Omega)$, \req{eq:fd_property_Linfty_RBLpaper} holds for every $\sigma\in H_+^2(\Omega)$ and $\kappa\in H^2 (\Omega)$, where $\mcF'(\sigma)$  considered as an operator from $H^2(\Omega)$ to $L^2(\Omega)$ is the \FD of $\mcF$.
 
To numerically solve \req{eq:inverse problem_RBLpaper}, we consider the nonlinear Landweber iteration that is based on the fix point equation 
\begin{eqnarray*}
\sigma = \xi(\sigma) := \sigma + \omega\mathcal{F}'(\sigma)^*(u - \mathcal{F}(\sigma)),
\end{eqnarray*}
where $\mathcal{F}'(\sigma)^*$ denotes the adjoint of $\mathcal{F}'(\sigma)$ and $\omega >0$ is a damping parameter.
With given noisy data $u^\delta\in L^2(\Omega)$ we can only expect to reconstruct an approximative solution $\sigma^\delta$ to an exact solution $\sigma^+\in H_+^2(\Omega)$ of \req{eq:inverse problem_RBLpaper}. The \emph{damped nonlinear Landweber iteration} is defined via
\begin{eqnarray}
\label{eq:Landweber_Iter_RBLpaper}	
\sigma_{n+1}^\delta = \sigma_n^\delta + \omega\mathcal{F}'(\sigma_n^\delta)^*(u^ \delta - \mathcal{F}(\sigma_n^\delta)),\quad n = 0,1,\dots
\end{eqnarray}
with starting value $\sigma_0^\delta\in H_+^2(\Omega)$, which may incorporate a-priori knowledge of $\sigma^+$, and damping parameter $\omega$ chosen as $\omega\leq \norm{\mcF'(\sigma^+)}^{-2}$. 
Since we consider noisy data, the iteration \req{eq:Landweber_Iter_RBLpaper} has to be stopped properly to prevent error amplification. We choose the well-known discrepancy principle: accept the iterate $\sigma_{n^*(\delta, u^\delta)}^\delta$ as a solution to \req{eq:inverse problem_RBLpaper}, if it fulfills
\begin{eqnarray}
\label{eq:discrepancy_principle_RBLpaper}
\fl
\norm{\mathcal{F}(\sigma_{n^*(\delta, u^\delta)}^\delta) - u^\delta}_{L^2(\Omega)} \leq \tau \delta \leq \norm{\mathcal{F}(\sigma_n^\delta) - u^\delta}_{L^2(\Omega)},\, n=0,1,\dots , n^*(\delta, u^\delta)-1,
\end{eqnarray}
with $\tau > 2$. In this setting the damped nonlinear Landweber iteration applied to \req{eq:inverse problem_RBLpaper} for noisy data is known to locally converge.

\begin{Proposition}
\label{Prop:Landweber_conv_RBLpaper}
Let $\sigma^+\in\mcD(\mcF)$ be a solution of $\mcF(\sigma) = u$. Then, there exists a radius $\rho > 0$ such that the following holds for every starting value $\sigma_0^\delta\in\mcB_\rho(\sigma^+)$: if the damped nonlinear Landweber iteration applied to noisy data $u^\delta\in L^2(\Omega)$ is stopped with $n^*(\delta, u^\delta)$ according to \req{eq:discrepancy_principle_RBLpaper}, then $\sigma_{n^*(\delta, u^\delta)}^\delta$ converges to some solution $\hat{\sigma}$ of $\mcF(\sigma) = u$ as $\delta\rightarrow 0$.
\end{Proposition}
\begin{proof}
Since $H^2(\Omega)$ embeds continuously into $L^\infty(\Omega)$ and therefore $H_+^2(\Omega)$ is open as mentioned in the beginning of the section, we can always find an open ball $\mcB_{r_1}(\sigma^+)\subset\mcD(\mcF) = H_+^2(\Omega)$ around $\sigma^+$ with $r_1 >0$ such that $\ei(\sigma) > c_1$ for all $\sigma\in \mcB_{r_1}(\sigma^+)$, where $c_1$ depends on $\ei(\sigma^+)$. Furthermore, the triangle inequality yields $\norm{\sigma}_{H^2(\Omega)}\leq \norm{\sigma - \sigma^+}_{H^2(\Omega)} + \norm{\sigma^+}_{H^2(\Omega)} < r_1 + \norm{\sigma^+}_{H^2(\Omega)} =: c_2$ for every $\sigma\in\mcB_{r_1}(\sigma^+)$. We mention that $H_+^2(\Omega)$ is convex. Hanke showed in \cite[Corollary $3.2$]{Hanke_97_a_example_2d} that
\begin{eqnarray*}
\fl
\| \mathcal{F}(\sigma) - \mathcal{F}(\tilde{\sigma}) - \mathcal{F}'(\tilde{\sigma})(\sigma-\tilde{\sigma})\|_{L^2(\Omega)} \leq C\norm{\sigma - \tilde{\sigma}}_{H^2(\Omega)}\norm{\mathcal{F}(\sigma) - \mathcal{F}(\tilde{\sigma})}_{L^2(\Omega)}
\end{eqnarray*}
holds for all $\sigma,\tilde{\sigma}\in \mcB_{r_1}(\sigma^+)$, where $C$ depends on $c_1,\,c_2$ and $\Omega$. Therefore, there exist $0 < r_2 \leq r_1$ and $\eta<\frac{1}{2}$ such that the \emph{tangential cone condition}
\begin{eqnarray*}
\| \mathcal{F}(\sigma) - \mathcal{F}(\tilde{\sigma}) - \mathcal{F}'(\tilde{\sigma})(\sigma-\tilde{\sigma})\|_{L^2(\Omega)} \leq \eta\|\mathcal{F}(\sigma) - \mathcal{F}(\tilde{\sigma})\|_{L^2(\Omega)}
\end{eqnarray*}
is true for all $\sigma,\tilde{\sigma}\in \mcB_{r_2}(\sigma^+)$. We now choose $\tau$ in \req{eq:discrepancy_principle_RBLpaper} as
\begin{eqnarray}
\label{eq:tau_choice_RBLpaper}
\tau > 2\frac{1+\eta}{1-\eta} > 2.
\end{eqnarray}
Finally, we can find $0 < r_3\leq r_2$ such that $\sqrt{\omega}\norm{\mathcal{F}'(\sigma)} \leq  \frac{\norm{\mcF'(\sigma)}}{\norm{\mcF'(\sigma^+)}} \leq 1$ for all $\sigma\in\mathcal{B}_{r_3}(\sigma^+)$. We now choose $\rho = \frac{r_3}{3}$ such that for every $\sigma_0^\delta\in \mcB_\rho(\sigma^+)$ all assumptions of \cite[Theorem $11.5$]{engl1996regularization} are fulfilled and the assertion follows.
\end{proof}
\begin{Remark}\hfill
\label{Remark:LW_conv_RBLpaper}
\begin{enumerate}[(i)]
\item Hanke \cite{Hanke_LW_13_tau1} extends the convergence result of \cite[Theorem $11.5$]{engl1996regularization} to choices $\tau > 1$. In the same article Hanke also mentions that the solution $\hat{\sigma}$ found with Theorem \ref{Prop:Landweber_conv_RBLpaper} depends on the starting value $\sigma_0^\delta$ and does not need to coincide with $\sigma^+$ if $F'(\sigma^+)$ happens to have a nontrivial null space.
\item Since in a practical application $\sigma^+$ is unknown, we will make a heuristic choice of the damping parameter $\omega$ in Section \ref{subsec:RBL_Exp_RBLpaper}.
\end{enumerate}
\end{Remark}


\section{Reduced basis methods}
\label{sec:chapter3_RBLpaper}
Before we introduce the key ingredients of the classical reduced basis method and for further numerical treatment, we discretize our model problem.

\subsection{Discretization}
\label{subsec:Discretization_RBLpaper}
We introduce a standard finite element space and a discrete parameter space. Note that the unit square, despite lacking a $C^2$-boundary, still meets the demands on the domain required for the theory in Section \ref{sec:Problem formulation_RBLpaper}. Therefore, we choose $\Omega := [0,1]^2$ as computational domain for the remainder of this article.

\begin{Definition}
\label{def:X_H_RBLpaper}
For a given $n\in\N,\,n\geq 2$, we choose a uniform triangulation of $\Omega$ with $(n+2)^2$ grid nodes $x_i$ and $I_{in}$ the index set of inner nodes. We use piecewise linear nodal basis functions, denoted as $\varphi_i,\, i \in I_{in}$, on the inner nodes. The discrete function space $Y$ then is defined via
\begin{eqnarray*}
Y := \{u:\Omega\rightarrow \R\mid u(x) = \sum_{i \in I_{in}} u_i\varphi_i(x),\, u_i\in\R,\, i \in I_{in}\}.
\end{eqnarray*}
$Y$ is equipped with the $L^2$-norm and and for $u\in Y$ let $\bi{u} = (u_i)_{i\in I_{in}}\in\R^{n^2}$ denote the vector of coefficients.
\end{Definition}

\begin{Definition}
\label{def:mcP_RBLpaper}
For a given square number $p=q^2,\,q\in\N$, we divide $\Omega$ into a uniform partition of $p$ squared subdomains $\Omega_i,\, i =1,\dots ,p$, and define $\mcP_p$ via
\begin{eqnarray*}
\mcP_p := \{\sigma:\Omega\rightarrow\R \mid \sigma(x) = \sum_{i =1}^p \sigma_i\, \chi_{\Omega_i}(x),\, \sigma_i \in \R_+ := (0,\infty),\, i=1,\dots p \},
\end{eqnarray*}
with $\chi_{\Omega_i}$ being the characteristic function on the subdomain $\Omega_i$. $\mcP_p$ is equipped with the $L^2$-norm and for $\sigma\in\mcP_p$ let $\bsigma = (\sigma_i)_{i=1}^p\in\Rpp$ denote the vector of coefficients. 
\end{Definition}

The following discrete problems will have $\Rpp$ and $\R^p$ as definition space. Note the isomorphisms between $\Rpp$ and $\mcP_p$ as well as $\R^p$ and its (analogously to Definition \ref{def:mcP_RBLpaper} defined) function space.
Furthermore, we recall that $b,f$ and $g$ are the bilinear and linear forms introduced in Section \ref{sec:Problem formulation_RBLpaper} and the stability constants of $b,f$ and $g$ carry over to $Y$ and $\Rpp$ such that existence and uniqueness of the discrete problems are guaranteed via Lax-Milgram. The \emph{discrete forward operator} is given by
\begin{eqnarray}
\label{eq:discfp_RBLpaper}
\eqalign{
F :\Rpp \longrightarrow Y,\quad\bsigma\longmapsto u^{\sigma},\,\bi{u} = \left(u_i^\sigma\right)_{i \in I_{in}}\mbox{ solving}\\
\bi{B}(\bsigma)\bi{u} = \bi{f}\mbox{ with }
(\bi{B}(\bsigma))_{ij} := b(\varphi_i , \varphi_j;\sigma),\,
(\bi{f})_i := f(\varphi_i),\, i,j \in I_{in}.}
\end{eqnarray}
The associated \textit{discrete inverse problem} is
\begin{eqnarray}
\label{eq:discip_RBLpaper}
\mbox{for }u^\sigma\in Y\mbox{ find }\bsigma\in \Rpp\mbox{ such that \req{eq:discfp_RBLpaper} is fulfilled}.
\end{eqnarray}
In the upcoming sections $u^\sigma$, the solution of \req{eq:discfp_RBLpaper}, will simply be denoted by $u$. For $\bsigma\in \Rpp$, the \emph{Jacobian} $F'(\bsigma)$ is given by
\begin{eqnarray}
\label{eq:discfd_RBLpaper}
\eqalign{
F'(\bsigma)(\cdot) :\R^p \longrightarrow Y,\quad\bkappa\longmapsto v_\kappa^\sigma,\,\bi{v} = \left(v_{\kappa,i}^\sigma\right)_{i \in I_{in}}\mbox{ solving}\\
\bi{B}(\bsigma)\bi{v} = \bi{g}(\bkappa)\mbox{ with }\bi{B}(\bsigma)\mbox{ as in \req{eq:discfp_RBLpaper} and }(\bi{g}(\bkappa))_i := g(\varphi_i;\kappa),\, i \in I_{in}.}
\end{eqnarray}
\begin{Remark}\hfill
\label{remark:discparamsapce_RBL_paper}
\begin{enumerate}[(i)]
\item The discrete setting introduced in this section deviates from the continuous setting introduced in Section \ref{sec:Problem formulation_RBLpaper}, where the forward operator $\mcF$ was a mapping from $H_+^2(\Omega)$ to $L^2(\Omega)$. Here, the discrete forward operator $F$ is a mapping from a finite subspace $\mcP_p$ of $L^2(\Omega)$ into another finite subspace $Y$ of $L^2(\Omega)$. This choice is made since it resembles the common numerical setting for the inverse problem tackled in this paper, where no continuity for the searched for diffusion coefficient can be assumed.
Do note that due to this choice the result of Proposition \ref{Prop:Landweber_conv_RBLpaper} does not need to hold in this discrete setting. Also, to our knowledge, it is an open question if the tangential cone condition required in the proof of Proposition \ref{Prop:Landweber_conv_RBLpaper} holds in this discrete setting.
\item Since we use the $L^2$-norm instead of the energy-norm on $Y$, it is $\alpha(\bsigma) := \frac{\min(\bsigma)}{C_{PF}^2} > 0$, for all $\sigma\in \mcP_p$, the coercivity constant of $b$ with respect to $Y$. For $\Omega = [0,1]^2$ we refer to the proof of \cite[Thm. $6.30$]{adams2003sobolev} and choose $C_{PF} = \frac{1}{\sqrt{2}}$ such that we use $\alpha(\bsigma) = 2\min(\bsigma) > 0$ for all $\sigma\in \mcP_p$ throughout this article.
\end{enumerate}
\end{Remark}

\subsection{The reduced basis method}
\label{subsec:RBM_RBLpaper}
Reduced basis methods aim at constructing a low-dimensional subspace $Y_N$ of $Y$, with $N = \dim Y_N \ll \dim Y = n^2$,  such that the reduced basis solution $u_N$ is an accurate approximation of $u$, the high-dimensional forward solution of \req{eq:discfp_RBLpaper}. Typically $Y_N$ will consist of so called \textit{snapshots} that are solutions of \req{eq:discfp_RBLpaper} to \emph{meaningful}  parameters. We will not discuss the construction of $Y_N$ in this section but assume a reduced basis space to be given. In order to give a brief overview of the reduced basis method, this section is kept very generic. For a detailed survey of the reduced basis method we refer to \cite{RB_master_paper, Ha14RBTut}. 

\begin{Definition}
\label{def:reducedprob_RBLpaper}
Let a forward operator \req{eq:discfp_RBLpaper} and a reduced basis space $Y_N \subset Y$, with $\dim Y_N = N$ and basis $\Psi_N := \{\psi_1, \dots , \psi_N\}$, be given. We define the \emph{discrete reduced forward operator}
\begin{eqnarray}
\label{eq:reduceddiscfp_RBLpaper}
\fl
\eqalign{
F_N :\Rpp \longrightarrow Y_N,\quad\bsigma\longmapsto u_N^{\sigma}= \sum_{i = 1}^N u_{N,i}^\sigma \psi_i,\,\bi{u}_N = (u_{N,i}^\sigma)_{i = 1}^N\mbox{ solving}\\
\bi{B}_N(\bsigma) \bi{u}_N = \bi{f}_N\mbox{ with }(\bi{B}_N(\bsigma))_{ij} := b(\psi_i , \psi_j;\sigma),\,
(\bi{f}_N)_i := f(\psi_i),\, i,j= 1,\dots ,N.}
\end{eqnarray}
We call $u_N^\sigma$ the \emph{reduced basis approximation} and will often write $u_N$ instead.
\end{Definition}
\begin{Remark}\hfill
\label{remark:redprob_RBLpaper}
\begin{enumerate}[(i)]
\item Existence and uniqueness of \req{eq:reduceddiscfp_RBLpaper} follow from the properties of \req{eq:discfp_RBLpaper}.
\item If the reduced basis $\Psi_N$ is orthonormal, $\mathrm{cond}(\bi{B}_N(\bsigma))\leq \frac{\gamma(\bsigma)}{\alpha(\bsigma)}$ holds independent of $N$ with $\bi{B}_N(\bsigma)$ defined in \req{eq:reduceddiscfp_RBLpaper}.
\end{enumerate}
\end{Remark}
For the sake of completeness, we include the proof of the well-known rigorous error estimator for the reduced basis error, here measured in the $L^2$-norm, $\norm{u - u_N}_{L^2(\Omega)}$, cf. \cite{RB_master_paper} or \cite[Proposition $2.15$ \& $2.19$]{Ha14RBTut}.
\begin{Lemma}
\label{Lemma:errest_RBLpaper}
For $\sigma\in\mcP_p$ we define the residual $r(\cdot ; \sigma)\in Y'$ via
\begin{eqnarray*}
r(v;\sigma) := f(v) - b(u_N,v;\sigma),\quad v\in Y.
\end{eqnarray*}
Then, let $v_r\in Y$ denote the Riesz-representative of $r(\cdot ;\sigma)$, i.e.,
\begin{eqnarray*}
\sp{v_r}{v}_{L^2(\Omega)} = r(v;\sigma),\, v\in Y,\quad \norm{v_r}_{L^2(\Omega)} = \norm{r(\cdot ; \sigma)}_{Y'}.
\end{eqnarray*}
Then, the error $u - u_N\in Y$ is bounded for all $\sigma\in\mcP_p$ by
\begin{eqnarray}
\label{eq:errest_RBLpaper}
\norm{u - u_N}_{L^2(\Omega)} \leq \Delta_N(\bsigma) := \frac{\norm{v_r}_{L^2(\Omega)}}{\alpha(\bsigma)}.
\end{eqnarray}
\end{Lemma}
\begin{proof}
Introducing the notation $e := u - u_N$ we can calculate
\begin{eqnarray*}
\fl
b(e,v;\sigma) &= b(u,v;\sigma) - b(u_N,v,\sigma) = f(v) - b(u_N,v;\sigma) = r(v;\sigma)\quad \mbox{for all }v\in Y.
\end{eqnarray*}
Testing this equation with $e\in Y$ yields
\begin{eqnarray*}
\fl
\alpha(\bsigma)\norm{e}_{L^2(\Omega)}^2 &\leq b(e,e;\sigma) = r(e;\sigma) \leq \norm{r(\cdot;\sigma)}_{Y'}\norm{e}_{L^2(\Omega)} = \norm{v_r}_{L^2(\Omega)} \norm{e}_{L^2(\Omega)}.
\end{eqnarray*}
Division by $\norm{e}_{L^2(\Omega)}$ and $\alpha(\bsigma)$ concludes the proof.
\end{proof}

We want to remind the reader that this is an estimator for the error between the reduced basis approximation and the discrete forward solution.
Since the construction method for $Y_N$ in Section \ref{sec:RBL_RBLpaper} will be snapshot-based, we note an important property of such methods, the \emph{reproduction of solutions}. It guarantees exactness in the reduced basis approximation for parameters whose snapshots are part of $Y_N$.
\begin{Lemma}
\label{Lemma:recon_of_sol_RBLpaper}
Let $\sigma\in\mcP_p,\, F(\bsigma),\, F_N(\bsigma)$ be solutions of \req{eq:discfp_RBLpaper} and \req{eq:reduceddiscfp_RBLpaper} and $\bi{e}_i\in\R^N$ the $i$-th unit vector. Then the following holds
\begin{enumerate}[(i)]
\item if $F(\bsigma)\in Y_N$ then $F_N(\bsigma) = F(\bsigma)$.
\item if $F(\bsigma) = \psi_i\in \Psi_N$ then $\bi{u}_N = \bi{e}_i\in\R^N$.
\end{enumerate}
\end{Lemma}
\begin{proof}
Immediately follows from \req{eq:discfp_RBLpaper} and \req{eq:reduceddiscfp_RBLpaper}, see, e.g., \cite[Proposition $2.16$]{Ha14RBTut}.
\end{proof}

To conclude this brief overview of the reduced basis method we present both offline/online decompositions of \req{eq:reduceddiscfp_RBLpaper} and the error estimator \req{eq:errest_RBLpaper} that allow for the rapid computation of $u_N$ and $\Delta_N$. The essential assumption for those decompositions is that the bilinear form $b$ and the linear form $f$ are \emph{parameter-separable}, which is fulfilled by \req{eq:discfp_RBLpaper}. 

\begin{Corollary}
\label{Cor:Affdecomp_RBLpaper}
Using the notation of Definition \ref{def:mcP_RBLpaper}, the set $\{\sigma^{(1)}(x),\dots , \sigma^{(p)}(x) \mid \sigma^{(i)}(x) = \chi_{\Omega_i}(x),\, i = 1,\dots ,p\}$ is a basis of $\mcP_p$ with corresponding coefficient vectors $\bsigma^{(i)} = \bi{e}_i\in\R^p,\, i = 1,\dots ,p$, $\bi{e}_i$ being the $i$-th unit vector. Therefore, we can rewrite $b$ and $f$ as
\begin{eqnarray*}
b(u,v;\sigma) = \sum_{q=1}^{Q_b}\Theta_b^q(\bsigma)b^q(u,v),\quad
f(v;\sigma) = \sum_{q=1}^{Q_f}\Theta_f^q(\bsigma)f^q(v),
\end{eqnarray*}
for all $u,v\in Y$ and $\sigma\in\mcP_p$, with $Q_b = p,\,Q_f = 1$ \emph{coefficient functions} $\Theta_b^q(\bsigma) := (\bsigma)_q,\, q = 1,\dots ,p,\,\Theta_f^1(\bsigma) := 1$ and \emph{components} 
\[
b^q(u,v) := b(u,v;\sigma^{(q)}),\,q = 1,\dots p,\,f^1(v) := f(v),\, u,v\in Y.
\]
If the bilinear form $b$ and the linear form $f$ can be rewritten this way, they are said to be \emph{parameter-separable}. 
Regarding the residual we set $Q_r := Q_f + N\cdot Q_b =  1 + N\cdot p$ and define the components of the residual $r^q\in Y',\, q = 1,\dots , Q_r$ via
\begin{eqnarray*}
\fl
\eqalign{
(r^1,\dots , r^{Q_r}) &:= (f^1(\cdot),\dots , f^{Q_f}(\cdot), b^1(\psi_1,\cdot),\dots ,b^{Q_b}(\psi_1,\cdot), \dots , b^1(\psi_N,\cdot), \dots ,b^{Q_b}(\psi_N,\cdot))\\
 &= (f^1(\cdot), b^1(\psi_1,\cdot),\dots ,b^{Q_b}(\psi_1,\cdot), \dots , b^1(\psi_N,\cdot), \dots ,b^{Q_b}(\psi_N,\cdot)),
}
\end{eqnarray*}
and let $v_r^q\in Y$ denote the Riesz-representative of $r^q$. 
For $u_N^{\sigma}= \sum_{i = 1}^N u_{N,i}^\sigma \psi_i$ a solution of \req{eq:reduceddiscfp_RBLpaper} we define the corresponding coefficient functions $\Theta_r^q(\bsigma),\, q = 1,\dots , Q_r$ via
\begin{eqnarray*}
\fl
\eqalign{
(\Theta_r^1,\dots ,\Theta_r^{Q_r}) &:= (\Theta_f^1 , \dots , \Theta_f^{Q_f},  -\Theta_b^1 u_{N,1}^\sigma, \dots , -\Theta_b^{Q_b} u_{N,1}^\sigma, \dots ,-\Theta_b^1 u_{N,N}^\sigma, \dots , -\Theta_b^{Q_b} u_{N,N}^\sigma)\\
&= (1, -(\bsigma)_1 u_{N,1}^\sigma, \dots , -(\bsigma)_p u_{N,1}^\sigma, \dots ,-(\bsigma)_1 u_{N,N}^\sigma, \dots , -(\bsigma)_p u_{N,N}^\sigma).
}
\end{eqnarray*}
Using this, the residual $r$ and its Riesz-representative $v_r$ are parameter-separable as
\begin{eqnarray*}
r(v;\sigma) = \sum_{q=1}^{Q_r}\Theta_r^q(\bsigma)r^q(v),\quad
v_r(\sigma) = \sum_{q=1}^{Q_r}\Theta_r^q(\bsigma)v_r^q.
\end{eqnarray*} 
\end{Corollary}

For problems that are not parameter-separable, the empirical interpolation method \cite{EIM} is available. 
The general idea of an offline/online decomposition is: compute all parameter independent quantities in a nonrecurring possibly expensive offline phase and then, for every new parameter, rapidly compute the desired quantity in the online phase. We first formulate the offline/online decomposition of \req{eq:reduceddiscfp_RBLpaper}.
\begin{Procedure}\hfill\\
\label{Proc:OffOnDecomp_RBLpaper}
\textit{Offline phase (one-time)}
\begin{enumerate}[(i)]
\item Generate reduced basis $\Psi_N = \{\psi_1,\dots ,\psi_N\}$ and $Y_N$.
\item Galerkin projection of components onto $Y_N$, i.e., compute $\bi{B}_N^q := \left(b^q(\psi_i,\psi_j)\right)_{i,j = 1}^N\in\mathbb{R}^{N\times N}$ and  $\bi{f}_N^q := \left(f^q(\psi_i)\right)_{i=1}^N\in\mathbb{R}^N$.
\end{enumerate}
\textit{Online phase (for each new $\sigma\in\mcP_p$)}
\begin{enumerate}[(i)]
\item Evaluate coefficient functions $\Theta_b^q(\bsigma)$, $\Theta_f^q(\bsigma)$, assemble $\bi{B}_{N}(\bsigma)$, $\bi{f}_{N}$ and solve linear system in \req{eq:reduceddiscfp_RBLpaper}.
\item Reconstruct reduced basis solution $u_N = \sum_{i=1}^N \bi{u}_{N,i}\psi_i$.
\end{enumerate}
\end{Procedure}

Since the online phase only involves linear combination and the solution of a linear system of dimension $N$, with $N \ll n^2$, it is very cheap. We conclude with the offline/online decomposition of the residual norm $\norm{v_r}_{L^2(\Omega)}$ and therefore the error estimator.
\begin{Procedure}\hfill\\
\label{Proc:OffOnDecompErrEst_RBLpaper}
\textit{Offline phase:} After the offline phase of Procedure \ref{Proc:OffOnDecomp_RBLpaper}, compute the matrix $\bi{G}_r := \left(\sp{v_r^q}{v_r^{q'}}_{L^2(\Omega)}\right)_{q,q' = 1}^{Q_r}$.\\
\textit{Online phase:} For given $\sigma\in \mcP_p$ and corresponding $u_N^\sigma$, evaluate $\bTheta_r := \left(\Theta_r^i(\bsigma)\right)_{i=1}^{Q_r}\in \R^{Q_r}$ and compute $\norm{v_r}_{L^2(\Omega)} = \sqrt{\bTheta_r\transp \bi{G}_r \bTheta_r}$.
\end{Procedure}


\section{Reduced Basis Landweber (RBL) method}
\label{sec:RBL_RBLpaper}
Before we develop the Reduced Basis Landweber (RBL) method, we introduce needed results and comment on the alternative direct approach that was mentioned in the Introduction of this paper.

\subsection{Preliminaries}
\label{subsec:Preliminaries_RBLpaper}
With the notation introduced in Section \ref{subsec:Discretization_RBLpaper} the nonlinear Landweber iteration defined in \req{eq:Landweber_Iter_RBLpaper} applied to \req{eq:discip_RBLpaper} is reasonable. As mentioned in Section \ref{sec:Problem formulation_RBLpaper} we consider the damped nonlinear Landweber iteration with damping parameter $\omega >0$ terminated with the discrepancy principle as it is stated in Algorithm \ref{Algo:LW_RBLpaper}.

\begin{algorithm}[h]
\caption{Landweber($\bsigma_{start},\tau$)}
\begin{algorithmic}[1]
\label{Algo:LW_RBLpaper}
\STATE{$n:=0,~\bsigma_0^\delta := \bsigma_{start}$}
\WHILE{$\norm{F(\bsigma_n^\delta) - u^\delta}_{L^2(\Omega)} > \tau\delta$}
\STATE{$\bsigma_{n+1}^\delta := \bsigma_n^\delta + \omega F'(\bsigma_n^\delta)^*(u^\delta - F(\bsigma_n^\delta))$}
\STATE{$n := n+1$}
\ENDWHILE
\RETURN{$\bsigma_{LW} := \bsigma_n^\delta$}
\end{algorithmic}
\end{algorithm}

In the upcoming sections we will write $\sigma_{LW}$ to denote the element in $\mcP_p$ corresponding to $\bsigma_{LW}$. 
We introduce a \emph{dual problem} that allows for a simple calculation of the Landweber update in line $3$ of Algorithm \ref{Algo:LW_RBLpaper}.
\begin{Proposition}
\label{Thm:FD_sp_eval_RBLpaper}
For $\bsigma\in\Rpp$, $\bkappa\in\R^p$ and $l\in Y$ it holds
\begin{eqnarray}
\label{eq:FDadj_eval_RBLpaper}
\sp{\bkappa}{F'(\bsigma)^*\,l}_2 = \sp{F'(\bsigma)\bkappa}{l}_{L^2(\Omega)} = \int_\Omega \kappa \nabla u^\sigma\cdot\nabla u_l^\sigma\, dx,
\end{eqnarray}
with $F'(\bsigma)^*$ the adjoint of $F'(\bsigma)$ and $u_l^\sigma\in Y$ the unique solution of the \emph{discrete dual problem}
\begin{eqnarray}
\label{eq:discadjprob_RBLpaper}
\eqalign{
\bi{B}(\bsigma)\bi{u} = \bi{m}(\bi{l})\mbox{ with  }
\bi{B}(\bsigma)\mbox{ as in \req{eq:discfp_RBLpaper} and}\\
(\bi{m}(\bi{l}))_i := m(\varphi_i; l) := -\int_\Omega \varphi_i\, l\,dx,\, i \in I_{in}.}
\end{eqnarray}
\end{Proposition}
\begin{proof}
Note that $u_l^\sigma\in Y$ solving \req{eq:discadjprob_RBLpaper} is equivalent to $u_l^\sigma$ solving $b(u_l^\sigma ,v;\sigma) = m(v)$ for all $v\in Y$. 
The first equality in \req{eq:FDadj_eval_RBLpaper} is the definition of the adjoint. The second equality follows from \req{eq:discfd_RBLpaper} and \req{eq:discadjprob_RBLpaper}
\begin{eqnarray*}
\sp{F'(\bsigma)\bkappa}{l}_{L^2(\Omega)} &= \int_\Omega F'(\bsigma)\bkappa\,l\,dx = \int_\Omega v_\kappa^\sigma\,l\, dx
  = -\int_\Omega\sigma\nabla u_l^\sigma\cdot\nabla v_\kappa^\sigma\,dx\\
 &= \int_\Omega \kappa \nabla u^\sigma\cdot \nabla u_l^\sigma\, dx.
\end{eqnarray*}
\end{proof}

Using Proposition \ref{Thm:FD_sp_eval_RBLpaper}, we calculate the Landweber update in line $3$ of Algorithm \ref{Algo:LW_RBLpaper}. 
For $\sigma\in\mcP_p$, $l\in Y$ let $u^\sigma = \sum_{i \in I_{in}} u_i^\sigma\varphi_i$ and $u_l^\sigma = \sum_{i \in I_{in}} u_{l,i}^\sigma\varphi_i\in Y$ be solutions of \req{eq:discfp_RBLpaper} and \req{eq:discadjprob_RBLpaper} with corresponding coefficient vectors $\bi{u}^\sigma = (u_i^\sigma)_{i\in I_{in}}$ and $\bi{u}_l^\sigma = (u_{l,i}^\sigma)_{i\in I_{in}}$ such that it holds for $\bkappa\in\R^p$
\begin{eqnarray}
\label{eq:discevalFDadj_RBLpaper}
\eqalign{
\sp{\bkappa}{F'(\bsigma)^*l}_2 &= \int_\Omega \kappa \nabla u^\sigma\cdot\nabla u_l^\sigma\, dx = \sum_{i,j\in I_{in}}u_i^\sigma u_{l,j}^\sigma \int_\Omega \kappa \nabla \varphi_i\cdot\nabla \varphi_j\, dx\\
 &= \sum_{i,j\in I_{in}}u_i^\sigma u_{l,j}^\sigma  b(\varphi_i , \varphi_j;\kappa) = (\bi{u}^\sigma)\transp\bi{B}(\bkappa) \bi{u}_l^\sigma.}
\end{eqnarray}

Therefore, we can evaluate $F'(\bsigma)^*l$ for given $\sigma\in\mcP_p,\, l\in Y$ by consecutively inserting a basis vector of $\R^p$ as the parameter $\bkappa$. Following Corollary \ref{Cor:Affdecomp_RBLpaper}, we choose the standard basis of $\mathbb{R}^p$ such that $\bi{B}(\bkappa)$ in \req{eq:discevalFDadj_RBLpaper} is the $k$-th component matrix $\bi{B}^k := \left(b^k(\varphi_i,\varphi_j)\right)_{i,j \in I_{in}}$ if $\bkappa$ is the $k$-th unit vector. Using this, the calculation of the Landweber update in line $3$ of Algorithm \ref{Algo:LW_RBLpaper} consists of the following steps.
\begin{Procedure}\hfill
\label{Proc:LW_update_RBLpaper}
\begin{enumerate}[1.]
\item Compute $u^{\sigma_n^\delta}$ the primal forward solution of \req{eq:discfp_RBLpaper}. Define $l :=  u^\delta - u^{\sigma_n^\delta}$.
\item Compute $u_{l}^{\sigma_n^\delta}$ the dual forward solution of \req{eq:discadjprob_RBLpaper}.
\item Evaluate the Landweber update
\begin{eqnarray*}
\left(F'(\bsigma_n^\delta)^*(u^\delta - F(\bsigma_n^\delta))\right)_k = (\bi{u}^{\sigma_n^\delta})\transp\bi{B}^k \bi{u}_{l}^{\sigma_n^\delta},\, k=1,\dots ,p.
\end{eqnarray*}
\end{enumerate}
\end{Procedure}

We conclude this preliminary section with a extensive comment on the alternative direct approach to couple reduced basis methods and the Landweber method.
\begin{Remark}
\label{remark:direct_approach}
Due to Procedure \ref{Proc:LW_update_RBLpaper}, every Landweber step contains two forward solutions and Algorithm \ref{Algo:LW_RBLpaper} as an iterative regularization algorithm provides a many-query context such that the application of reduced basis methods is intuitive. The direct approach consists of constructing one \emph{global reduced basis space}, yielding accurate reduced basis approximations for all $\sigma\in\mcP_p$, per forward problem and replacing the corresponding forward solution required in the Landweber iteration with its reduced counterpart. We note that this methodology surely could be applied to other regularization algorithms as well and that similar techniques have successfully been applied to problems with a low-dimensional parameter space, see, e.g., \cite{nguyenrozza2009reduced, nguyenPHDThesis, HoangDentalTissuereduced}. The application of this direct approach to our model problem is limited for two reasons. First, concerning \req{eq:discfp_RBLpaper}, it is possible to construct a global reduced basis space via e.g. the well-known greedy algorithm, see, e.g., \cite{veroy2003posteriori, RB_master_paper, Ha14RBTut}. Concerning \req{eq:discadjprob_RBLpaper}, due to the varying right hand side, this is not possible. Therefore, we could only speed up one of the two forward solutions such that this approach would be inefficient for the chosen model problem. Second, global reduced basis spaces can only be constructed when the parameter space is bounded and of low dimension (say up to $100$, where in \cite{Stamm_Efficient_greedy, eftang2010hp, haasdonk2011training, urban2014greedy} various methods to construct such spaces are presented) and therefore it is not feasible for our purpose that lies in imaging.
\end{Remark}

\subsection{Development of the method}
\label{subsec:the_method_RBLpaper}
One way to couple reduced basis methods and the nonlinear Landweber method is the direct approach that was discussed in Remark \ref{remark:direct_approach} which had several drawbacks. In this section we want to develop a new method that overcomes these drawbacks via  adaptive online updates of the reduced basis space during the solution process of the inverse problem. Online updates in model order reduction have also been considered recently in other contexts \cite{Druskin_Zaslavski, cuiWilcox2014datadriveninversion, lass2014PHDThesis, zahr2014progressive}.

We no longer aim at constructing a global reduced basis space that could be used for the reconstruction of every $\sigma\in\mcP_p$. Instead, for a given measurement, we aim at adaptively constructing a small problem-oriented reduced basis space $Y_{N,1}$ while also solving the associated inverse problem. Therefore,  $Y_{N,1}$ aims only at a specific yet unknown region of $\mcP_p$ that is relevant for the solution of the inverse problem. This breaks the typical offline/online framework of reduced basis methods but the resulting method will still have offline and online segments. Nevertheless, the procedure provides considerable acceleration of the computational time.

For the construction of this problem-oriented space we use the nonlinear Landweber method projected onto the current reduced basis space as a criterion to select meaningful parameters. These are then used to enrich the reduced basis space with the corresponding snapshot. 
By this choice we construct a reduced basis space that is tailored around the inverse problem in the sense that it provides accurate reduced basis approximations for parameters lying in the a priori unknown region of the parameter space that is relevant for the solution of the inverse problem.
Simultaneously the inverse problem is solved in this process.
 Since the Landweber method makes use of the adjoint of the derivative, we introduce a second reduced basis space $Y_{N,2}$ containing the required information. We gather these thoughts.

\begin{Procedure}\hfill
\label{Proc:RBL_RBLpaper}
\begin{enumerate}[1.]
\item Start with an initial guess $\sigma_{start}\in\mcP_p$ and initial possibly empty spaces $Y_{N,1},\,Y_{N,2}$.
\item Update the spaces $Y_{N,1},\,Y_{N,2}$ using the current iterate. In the first step use $\sigma_{start}$.
\item Solve the inverse problem up to a certain accuracy with the nonlinear Landweber method projected onto $Y_{N,1}$ and $Y_{N,2}$ and thus determine a new parameter for space enrichment.
\item If the current iterate fulfills \req{eq:discrepancy_principle_RBLpaper}, terminate, else go to step $2$.
\end{enumerate}
\end{Procedure}

For now, we treat the update of the reduced spaces and the projected Landweber method as modular blocks of our procedure and elaborate on these after the final algorithm has been presented. First, we need to find meaningful termination criteria for the projected nonlinear Landweber method in step $3$ of Procedure \ref{Proc:RBL_RBLpaper}. Let in the following  $\sigma\in\mcP_p$ be the current iterate of the projected method.

Terminating with a high-dimensional discrepancy principle as soon as $\norm{F(\bsigma) - u^\delta}_{L^2(\Omega)} \leq \tau\delta$ is out of question since we do not want to compute the expensive solution of \req{eq:discfp_RBLpaper} in each iteration. Instead, we want to terminate via a low-dimensional discrepancy principle as soon as $\norm{F_N(\bsigma) - u^\delta}_{L^2(\Omega)} \leq \tau\delta$. Taking a look at
\begin{eqnarray}
\label{eq:highDP_success_RBLpaper}
&\norm{F(\bsigma) - u^\delta}_{L^2(\Omega)} \leq \norm{F(\bsigma) - F_N(\bsigma)}_{L^2(\Omega)} + \norm{F_N(\bsigma) - u^\delta}_{L^2(\Omega)},\\
\label{eq:redDP_failure_RBLpaper}
&\norm{F_N(\bsigma) - u^\delta}_{L^2(\Omega)} - \norm{F_N(\bsigma) - F(\bsigma)}_{L^2(\Omega)} \leq \norm{F(\bsigma) - u^\delta}_{L^2(\Omega)},
\end{eqnarray}
we can see that $\norm{F(\bsigma) - u^\delta}_{L^2(\Omega)}$ and $\norm{F_N(\bsigma) - u^\delta}_{L^2(\Omega)}$ are connected via the reduced basis error $\norm{F_N(\bsigma) - F(\bsigma)}_{L^2(\Omega)}$ that will grow over the course of the projected Landweber iteration since each consecutive iterate will be worse and worse approximated by the current set of reduced basis spaces. Therefore, we want to control this error which can be done using the rigorous error estimator $\Delta_N$ introduced in Lemma \ref{Lemma:errest_RBLpaper}. As long as $\sigma$ does not fulfill the reduced discrepancy principle $(\tau - 2)\delta$ is a reasonable upper bound for $\Delta_N$ since it follows from \req{eq:redDP_failure_RBLpaper} that $\norm{F(\bsigma) - u^\delta}_{L^2(\Omega)} > 2\delta$ (and therefore $\sigma$ is rejected by \req{eq:discrepancy_principle_RBLpaper} as well) as long as $\Delta_N(\bsigma) \leq (\tau -2)\delta$. This is a strong motivation to suggest the termination of step $3$ of Procedure \ref{Proc:RBL_RBLpaper} if one of the following criteria is met
\begin{eqnarray}
\label{eq:RBL_alt_termination_RBLpaper}
\norm{F_N(\bsigma) - u^\delta}_{L^2(\Omega)}\leq \tau\delta\quad\mbox{or}\quad\Delta_N(\bsigma) > (\tau -2)\delta.
\end{eqnarray}

The latter \emph{alternative termination criterion} in \req{eq:RBL_alt_termination_RBLpaper} is in fact a trust region criterion: as soon as the error estimator grows too large we cannot ensure that the error $\norm{F_N(\bsigma) - F(\bsigma)}_{L^2(\Omega)}$ stays small enough, thus we do not trust the current reduced basis spaces anymore (they might not produce feasible approximations anymore such that further iterations might be misleading) and enrich them using the current iterate.

We add these thoughts to Procedure \ref{Proc:RBL_RBLpaper} and call the resulting new method \emph{Reduced Basis Landweber (RBL) method}, see Algorithm \ref{Algo:RBL_RBLpaper}. 

\begin{algorithm}[h]
\caption{RBL($\bsigma_{start},\tau,\Psi_{N,1} ,\Psi_{N,2}$)}
\begin{algorithmic}[1]
\label{Algo:RBL_RBLpaper}
\STATE{$n:=0,~\bsigma_0^\delta := \bsigma_{start},~Y_{N,1} := \mbox{span}(\Psi_{N,1}),~Y_{N,2} := \mbox{span}(\Psi_{N,2})$}
\WHILE{$\norm{F(\bsigma_n^\delta) - u^\delta}_{L^2(\Omega)} > \tau\delta$}
\STATE{compute $\psi_{n,2}$ as described in \req{eq:RBL_sensitivityupdate_RBLpaper}}
\STATE{$\Psi_{N,1} := \Psi_{N,1}\cup \{F(\bsigma_n^\delta)\},~\Psi_{N,2} := \Psi_{N,2}\cup \{\psi_{n,2}\}$}
\STATE{$Y_{N,1} = \mbox{span}\{\Psi_{N,1}\},~Y_{N,2} = \mbox{span}\{\Psi_{N,2}\}$}
\STATE{$i:=1,~\bsigma_i^\delta := \bsigma_n^\delta$}
\REPEAT
\STATE{compute $s_{n,i}$ as described in Procedure \ref{Proc:RBL_update_RBLpaper}}
\STATE{$\bsigma_{i+1}^\delta := \bsigma_i^\delta + \omega s_{n,i}$}
\STATE{$i := i+1$}
\UNTIL{$\norm{F_N(\bsigma_i^\delta) - u^\delta}_{L^2(\Omega)} \leq \tau\delta$ }\OR{ $\Delta_N(\bsigma_i^\delta) > (\tau -2)\delta$}
\STATE{$\bsigma_{n+1}^\delta := \bsigma_i^\delta$}
\STATE{$n := n+1$}
\ENDWHILE
\RETURN{$\bsigma_{RBL} := \bsigma_n^\delta$}
\end{algorithmic}
\end{algorithm}

\begin{Remark}\hfill
\label{remark:RBL_RBLpaper}
\begin{enumerate}[(i)]
\item The reduced bases $\Psi_{N,1},\,\Psi_{N,2}$ are orthonormalized to ensure numerical stability according to Remark \ref{remark:redprob_RBLpaper}.
\item Computing $\Delta_N(\bsigma_i^\delta)$ is crucial regarding the total computational time of Algorithm \ref{Algo:RBL_RBLpaper}. We will elaborate in Section \ref{subsec:RBL_Exp_RBLpaper}.
\item The alternative termination criterion in \req{eq:RBL_alt_termination_RBLpaper} guarantees that the reduced basis error stays very small. Due to \req{eq:highDP_success_RBLpaper}, we expect Algorithm \ref{Algo:RBL_RBLpaper} to terminate as soon as the inner repeat loop terminates with $\norm{F_N(\bsigma_i^\delta) - u^\delta}_{L^2(\Omega)} \leq \tau\delta$. 
\item A possible drawback of the alternative termination criterion could be the error estimator being inefficient, i.e., in the notation of Lemma \ref{Lemma:errest_RBLpaper}, $\Delta_N(\bsigma_i^\delta)/\norm{e}_{L^2(\Omega)}$ being large. This could result in a premature termination of the repeat loop, wasting possible cheap repeat loop iterations and possibly causing more than necessary expensive while loop iterations.
\item Analogous to Algorithms \ref{Algo:LW_RBLpaper} we write $\sigma_{RBL}$ to denote the element in $\mcP_p$ corresponding to $\bsigma_{RBL}$.
\end{enumerate}
\end{Remark}

We want to elaborate on line $3$ and $8$ of Algorithm \ref{Algo:RBL_RBLpaper} with respect to our chosen model problem. Regarding the space update of $Y_{N,2}$, we refer to  Procedure \ref{Proc:LW_update_RBLpaper} and choose snapshots of the dual problem for the basis update $\psi_{n,2}$ of $Y_{N,2}$
\begin{eqnarray}
\label{eq:RBL_sensitivityupdate_RBLpaper}
\mbox{enrich } \Psi_{N,2} \mbox{ with }\psi_{n,2} = u_l^{\sigma_n^\delta}\mbox{ solving \req{eq:discadjprob_RBLpaper} for }l := u^\delta - F(\bsigma_n^\delta).
\end{eqnarray}
The \emph{reduced Landweber update} in line $8$ of Algorithm \ref{Algo:RBL_RBLpaper} is done along the lines of Procedure \ref{Proc:LW_update_RBLpaper} as well: for given spaces $Y_{N,1} = \mbox{span}\{\psi_{1,1},\dots , \psi_{1,N_1 }\}$, $Y_{N,2} = \mbox{span}\{\psi_{2,1},\dots , \psi_{2,N_2 }\}$ and current iterate $\bsigma_i^\delta$, we replace the forward solutions of \req{eq:discfp_RBLpaper} and \req{eq:discadjprob_RBLpaper} with their reduced counterpart. This is summarized in the following Procedure.
\begin{Procedure}\hfill
\label{Proc:RBL_update_RBLpaper}
\begin{enumerate}[1.]
\item Compute $F_N(\sigma_i^\delta) = u_N^{\sigma_i^\delta}$ the primal reduced basis approximation via \req{eq:reduceddiscfp_RBLpaper} using $Y_{N,1}$ as reduced basis space. Define $l :=  u^\delta - u_N^{\sigma_i^\delta}$.
\item Compute the dual reduced basis approximation $u_{N,l}^{\sigma_i^\delta} = \sum_{j = 1}^{N_2}u_{N,l,j}^{\sigma_i^\delta}\psi_{2,j}$ with $\bi{u}_{N,l}^{\sigma_i^\delta} = (u_{N,l,j}^{\sigma_i^\delta})_{j = 1}^{N_2}$ solving the small linear system
\begin{eqnarray}
\label{eq:reduceddiscadjprob_RBLpaper}
\fl
\eqalign{
\tilde{\bi{B}}_N(\bsigma_i^\delta) \bi{u}_{N,l}^{\sigma_i^\delta} &= \tilde{\bi{m}}_N(\bi{l})\mbox{ with }\\
(\tilde{\bi{B}}_N(\bsigma_i^\delta))_{j,k} &:= b(\psi_{2,j} , \psi_{2,k};\sigma_i^\delta),\,
(\tilde{\bi{m}}_N(\bi{l}))_j := m(\psi_{2,j};l),\, j,k = 1,\dots , N_2.}
\end{eqnarray}
\item Evaluate the \emph{reduced Landweber update} $s_{n,i}$ via 
\begin{eqnarray*}
\label{eq:RBL_update_RBLpaper}
\left(s_{n,i}\right)_k = (\bi{u}_N^{\sigma_i^\delta})\transp\bi{Q}^k \bi{u}_{N,l}^{\sigma_i^\delta},\, k=1,\dots ,p,
\end{eqnarray*}
with $\bi{Q}^k = b^k(\psi_{1,i} , \psi_{2,j})_{i,j = 1}^{N_1,N_2}\in\R^{N_1\times N_2}$ being parameter independent.
\end{enumerate}
\end{Procedure}

Since both reduced problems \req{eq:reduceddiscfp_RBLpaper} and \req{eq:reduceddiscadjprob_RBLpaper} are offline/online decomposable, i.e., the online phases are of complexity polynomial in $N_1$ and $N_2$, independent of $n^2$, and the matrices $\bi{Q}^k,\,k=1,\dots p$, can be computed as soon as the reduced spaces are updated in line $4$, the projected reduced Landweber method in the repeat loop from line $7$ to $10$ can be implemented in an efficient and cheap way. It only consists of the online phases of \req{eq:reduceddiscfp_RBLpaper} and \req{eq:reduceddiscadjprob_RBLpaper} and the reduced Landweber update in step $3$ of Procedure \ref{Proc:RBL_update_RBLpaper}. In this sense the repeat loop is the online segment of Algorithm \ref{Algo:RBL_RBLpaper}, where we elaborate on the computational cost of the error estimator $\Delta_N$ in the upcoming section. The remainder of the algorithm is then the offline segment since, with the enrichment of the reduced basis spaces, i.e., computing solutions of \req{eq:discfp_RBLpaper} and \req{eq:discadjprob_RBLpaper}, and the projection onto the new set of reduced basis spaces, it involves computations depending on $n^2$.

We conclude this section with final remarks about the RBL method.
\begin{Remark}\hfill
\label{remark:notheory_for_RBL_RBLpaper}
\begin{enumerate}[(i)]
\item For a fixed $\sigma\in\mcP_p$ let $\sp{\cdot}{\cdot}_\sigma := b(\cdot,\cdot;\sigma)$ denote the energy scalar product and $P_\sigma : Y \longrightarrow Y_{N,1}$ the corresponding orthogonal projection. With $Y_{N,1} = \mbox{span}\{\psi_{1,1},\dots , \psi_{1,N_1 }\}$ it holds for all $i = 1,\dots , N_1$
\begin{eqnarray*}
\eqalign{
&\sp{P_\sigma(F(\sigma)) - F(\sigma)}{\psi_{1,i}}_\sigma = 0\\
\Leftrightarrow \,&b(P_\sigma(F(\sigma)) - F(\sigma),\psi_{1,i};\sigma) = 0\\
\Leftrightarrow \,&b(P_\sigma(F(\sigma)),\psi_{1,i};\sigma) = f(\psi_{1,i})
}
\end{eqnarray*}
such that $P_\sigma\circ F(\sigma)$ is a solution of \req{eq:reduceddiscfp_RBLpaper} and therefore $u_N^\sigma = F_N(\sigma) = P_\sigma\circ F(\sigma)$, since the solution of \req{eq:reduceddiscfp_RBLpaper} is unique. Using this and the fact that $P_\sigma\in L(Y,Y_{N,1})$ is an orthogonal projection, it is easy to see that $F'_N(\sigma) = P_\sigma\circ F'(\sigma)$ and (for $l\in Y_{N,1}$) $F'_N(\sigma)^*l = F'(\sigma)^*(P_\sigma^*l) = F'(\sigma)^* l$. Therefore, the reduced Landweber update $s_{n,i}$ calculated in step $3$ of Procedure \ref{Proc:RBL_update_RBLpaper} does not coincide with the expression $F'_N(\bsigma_i^\delta)^*(u^\delta - F_N(\bsigma_i^\delta))$ (see Procedure \ref{Proc:LW_update_RBLpaper}). The consequence here is, that in general there is no closed expression of the iteration scheme of the RBL method. Instead, the main idea of the RBL method is to determine what kind of PDE solutions are required for the update of the Landweber method and replace them with suitable reduced basis approximations.
\item In \cite{Scherzer_iterativ_multilevel} Scherzer proposes a different methodology where a sequence $\{X_N\}_{N\in\N_0}$ of nested subspaces of $X$ (the infinite dimensional parameter space) with $\bigcup_{N\in\N_0} X_N$ being dense in $X$ is employed. The orthonormal projection $P_N$ on those spaces is then used to develop a multi-level discrete Landweber method. In \cite{Scherzer_et_al_multilvl_Banach} a similar approach is made for a steepest descent method in Banach spaces. As we can see our approach can not be formulated in such a way and the convergence theory developed in the mentioned works can not be adapted to the RBL method.
\item Because of $(i)$ and $(ii)$, we postpone a theoretical investigation of the RBL method to future work.
\end{enumerate}

\end{Remark}


\subsection{Experiments}
\label{subsec:RBL_Exp_RBLpaper} 
 
We want to compare Algorithms \ref{Algo:LW_RBLpaper} \& \ref{Algo:RBL_RBLpaper}. To this end we choose a specific setting to which all experiments refer. We use the parameter space $\mcP_{900}$, $n = 149$ for the finite element space $Y$ and want to reconstruct
\begin{eqnarray*}
\fl
\sigma^+(\bi{x}):= 3 + 2\chi_{\Omega^{(1)}}(\bi{x}) - 2\chi_{\Omega^{(2)}}(\bi{x})\quad \bi{x}\in\Omega\mbox{ with subdomains }\\
\fl
\Omega^{(1)} = [5/30 , 9/30]\times [3/30 , 27/30] \cup \left([9/30 , 27/30] \times \left([3/30 , 7/30] \cup [23/30 , 27/30]\right)\right),\\
\fl
\Omega^{(2)} = \{ \bi{x}\in\Omega\mid \norm{(18/30 , 15/30)\transp - \bi{x}}_2 < 4/30\}.
\end{eqnarray*}
 This is a piecewise constant function with background $3$, contrast $2$ on the C-shaped subdomain $\Omega^{(1)}$ and contrast $-2$ on the disk $\Omega^{(2)}$.
The starting value $\sigma_{start} \equiv 3$ as well as the function $\sigma^+$  are visualized in the top left and bottom left of Figure \ref{fig:compare_reconstructions_RBLpaper}. 
The noisy measurement $u^\delta$ is generated in the following way: we compute the PDE-solution for $\sigma^+$ using Comsol\textsuperscript{\textregistered} and evaluate it at the nodes of the finite element space $Y$. Afterwards we add uniformly distributed random noise with a certain noise level. If not specified differently we add $1\%$ relative noise (corresponding to $1.243\cdot 10^{-4}$ absolute noise) and choose $\tau = 2.5$ in this section. Note that the noise is only added on the inner nodes of the discretization since we assume that the homogeneous Dirichlet data are known and measured correctly. The damping parameter $\omega$ is heuristically chosen as $\omega = \frac{1}{2}(\|F'(\bsigma_{start})\|)^{-1}$ with the $\frac{1}{2}$ resembling local uniform boundedness. Note that $\|F'(\bsigma_{start})\|\ll 1$ such that $\omega$ actually serves as a speed-up of the iteration. The numerical experiments are done using  Matlab\textsuperscript{\textregistered} in conjunction with the libraries RBmatlab and KerMor, which both can be found online\footnote{http://www.ians.uni-stuttgart.de/MoRePaS/software/}.

In the following, three experiments will be carried out in the above \emph{full} setting.  Additionally, the same experiments are done for the (accordingly modified) versions of Algorithms \ref{Algo:LW_RBLpaper} \& \ref{Algo:RBL_RBLpaper} applied to the partial inverse problem \req{eq:partial_inverse problem_RBLpaper} introduced in Remark \ref{remark:PDE_RBLpaper}. In this \emph{partial} setting, we  use the same numerical setting as above but measure the data only on the subdomain $\tilde{\Omega} = [0,1]\times[0,0.5]$, add $1\%$ relative noise (corresponding to $8.786\cdot 10^{-5}$ absolute noise) and use $\tilde{\omega} = \frac{1}{2}(\|\tilde{F}'(\bsigma_{start})\|)^{-1}$ as damping parameter. Figures \ref{fig:compare_reconstructions_RBLpaper} \& \ref{fig:RBL_vs_LW_Updates_RBLpaper} as well as Table \ref{tab:compare_RBL_LW_RBLpaper} contain results for both (partial and full) settings where our discussion will focus on the full setting.

Figure \ref{fig:compare_reconstructions_RBLpaper} shows the reconstructions of $\sigma^+$ in the full setting via Algorithm \ref{Algo:RBL_RBLpaper} in the top middle and via Algorithm \ref{Algo:LW_RBLpaper} in the bottom middle. In addition the reconstructions in the partial setting via Algorithm \ref{Algo:RBL_RBLpaper} in the top right and via Algorithm \ref{Algo:LW_RBLpaper} in the bottom right are shown, where the black box indicates the subdomain $\tilde{\Omega}$.

\begin{figure}[h]
\centering\includegraphics[width=0.9\textwidth]{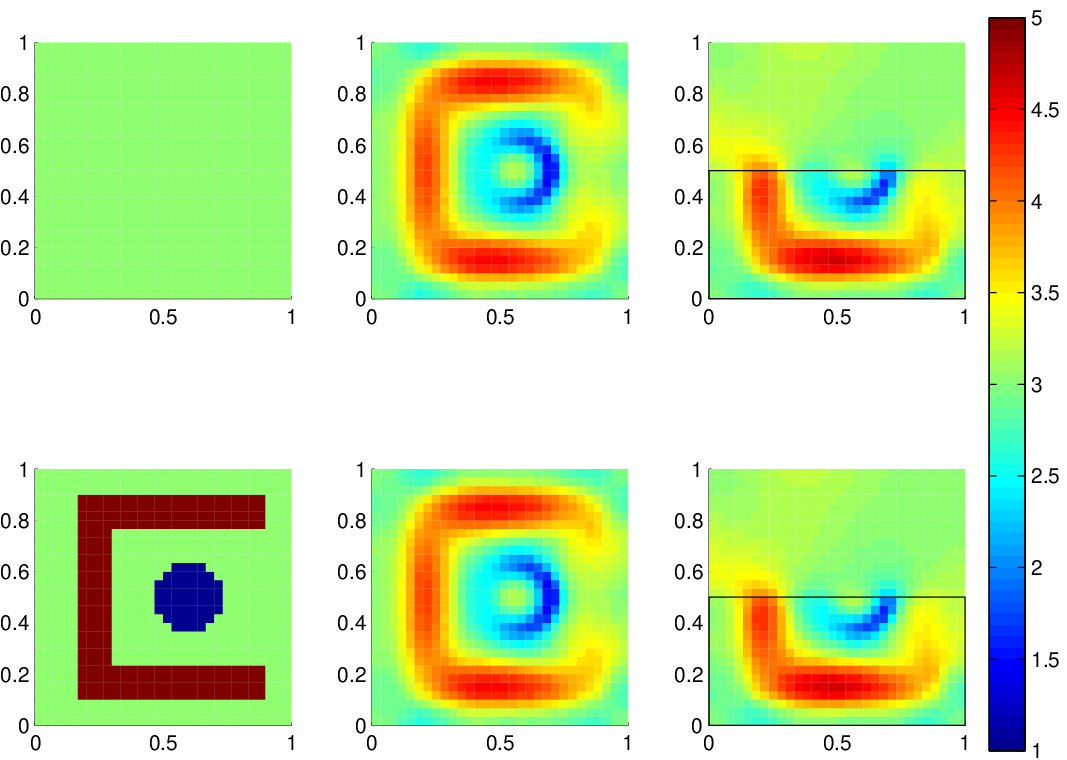}
\caption{The reconstructions of $\sigma^+$ in the full setting via Algorithm \ref{Algo:RBL_RBLpaper} (top middle) and Algorithm \ref{Algo:LW_RBLpaper} (bottom middle), as well as the starting value $\sigma_{start}$ (top left) and the exact value $\sigma^+$ (bottom left). The reconstruction in the partial setting via Algorithm \ref{Algo:RBL_RBLpaper} (top right) and Algorithm \ref{Algo:LW_RBLpaper} (bottom right) including the subdomain $\tilde{\Omega}$ as a black box.}
\label{fig:compare_reconstructions_RBLpaper}
\end{figure}

Concerning the middle column, we cannot distinguish the two reconstructions visually from each other, which is also stated by $\norm{\sigma_{RBL} - \sigma_{LW}}_{L^2(\Omega)} \approx 1.118\cdot 10^{-5}$, such that both algorithms numerically yield the same reconstruction. The shape and location of as well as the contrast on $\Omega^{(1)}$ are well reconstructed. Regarding  $\Omega^{(2)}$, only the location is well reconstructed. The contrast is not fitting everywhere and there is another small circular inclusion with opposite sign inside of $\Omega^{(2)}$. In the partial setting we have a good reconstruction on $\tilde{\Omega}$ and some indications of a reconstruction close to $\tilde{\Omega}$.

Next, we want to compare the Algorithms with respect to the computational time: for Algorithm \ref{Algo:LW_RBLpaper} we measure the total time, the amount of iterations until the discrepancy principle is reached and therefore the time per iteration, as well as the total amount of forward solves. Due to Procedure \ref{Proc:LW_update_RBLpaper}, both discretized problems \req{eq:discfp_RBLpaper} and \req{eq:discadjprob_RBLpaper} are considered here. For Algorithm \ref{Algo:RBL_RBLpaper}, we are interested in the total time and therefore the speed up compared to Algorithm \ref{Algo:LW_RBLpaper}, the amount of and the time per outer iteration (line $2$ - $14$ excluding the repeat-loop from line $7$ - $11$), as well as the amount of and the time per inner iteration (one step of the repeat-loop from line $7$ - $11$) and again the total amount of high-dimensional forward solves. Table \ref{tab:compare_RBL_LW_RBLpaper} contains the respective information.

\begin{table}[h]
\centering
\begin{tabular}{|c|c|c|c|c|c|c|}
\hline
Setting & time (s) & \multicolumn{2}{|c|}{\# Iterations} & \multicolumn{2}{|c|}{time per Iter. (s)} & \# forward solves\\
 \hline
 LW, full  & 187189 & \multicolumn{2}{|c|}{608067} & \multicolumn{2}{|c|}{0.308} & 1216134 \\
 \hline
 & & outer & inner & outer & inner & \\
\hline
 RBL, full & 14661 & 20 & 608083 & 3.705 & 0.024 & 40 \\
\hline
 LW, partial  & 173759 & \multicolumn{2}{|c|}	{580129} & \multicolumn{2}{|c|}{0.299} & 1160258 \\
 \hline
 & & outer & inner & outer & inner & \\
\hline
 RBL, partial & 10638 & 20 & 580133 & 3.670 & 0.018 & 40 \\
\hline
\end{tabular}
\caption{Runtime comparison of Algorithms \ref{Algo:LW_RBLpaper} \& \ref{Algo:RBL_RBLpaper} in the full and partial setting.}
\label{tab:compare_RBL_LW_RBLpaper}
\end{table}

In the full setting the RBL method needs around $4$ hours and the Landweber method needs around $52$ hours of computational time resulting in a speed-up of $13$. 
Due to \req{eq:RBL_alt_termination_RBLpaper}, the reduced spaces are very accurate such that the amount of inner iterations of Algorithm \ref{Algo:RBL_RBLpaper} roughly coincides with the amount of iterations of Algorithm \ref{Algo:LW_RBLpaper}.
We want to highlight that Algorithm \ref{Algo:RBL_RBLpaper} only needed $40$ expensive forward solves compared to the $1216134$ forward solves required for Algorithm \ref{Algo:LW_RBLpaper}.
If we look at the average iteration times of Algorithm \ref{Algo:RBL_RBLpaper}, it becomes clear that sufficient inner iterations have to be made per outer iteration for Algorithm \ref{Algo:RBL_RBLpaper} to pay off in time. We will see in our next experiment that this is the case in the chosen setting. 
Regarding the average time per inner iteration, we have to mention our implementation of the error estimator \req{eq:errest_RBLpaper}: we do not use the offline/online decomposition presented in Procedure \ref{Proc:OffOnDecompErrEst_RBLpaper} since this would result in each online phase to contain a vector-matrix-vector multiplication of dimension $Q_r = N_1\cdot p + 1$, with $p=900$ and $N_1 = \dim Y_{N,1}$. But the matrix in this multiplication is full as we can see in Corollary \ref{Cor:Affdecomp_RBLpaper}, which prohibits this approach. Therefore, we compute the Riesz-representative and its norm in each inner iteration according to Lemma \ref{Lemma:errest_RBLpaper} such that the online segment of Algorithm \ref{Algo:RBL_RBLpaper} is not completely independent of $n^2$ in this example and roughly $50\%$ of the total computational time is spent in computing the error estimator in the inner loop. Hence, it might be interesting to develop termination criteria that still guarantee accurate reduced basis spaces but are less expensive. Similar conclusions can be drawn for the partial setting as it can be seen in Table \ref{tab:compare_RBL_LW_RBLpaper}.

In the third experiment we want to see that both Algorithms \ref{Algo:LW_RBLpaper} \& \ref{Algo:RBL_RBLpaper} behave the same way, with the latter simply being faster. To this end we define the nonlinear Landweber update $\bi{s}_{n,LW} := F'(\bsigma_n^\delta)^*(u^\delta - F(\bsigma_n^\delta))$ in addition to the update of the RBL method (here denoted by $s_{n,RBL}$) described in Procedure \ref{Proc:RBL_update_RBLpaper}. Figure \ref{fig:RBL_vs_LW_Updates_RBLpaper} shows the \emph{update error} $\norm{s_{n,RBL} - s_{n,LW}}_{L^2(\Omega)}$ over the course of the iteration. The plot also includes a vertical dashed line whenever an outer iteration in Algorithm \ref{Algo:RBL_RBLpaper} is performed.

\begin{figure}[h]
\centering\includegraphics[scale = 0.6]{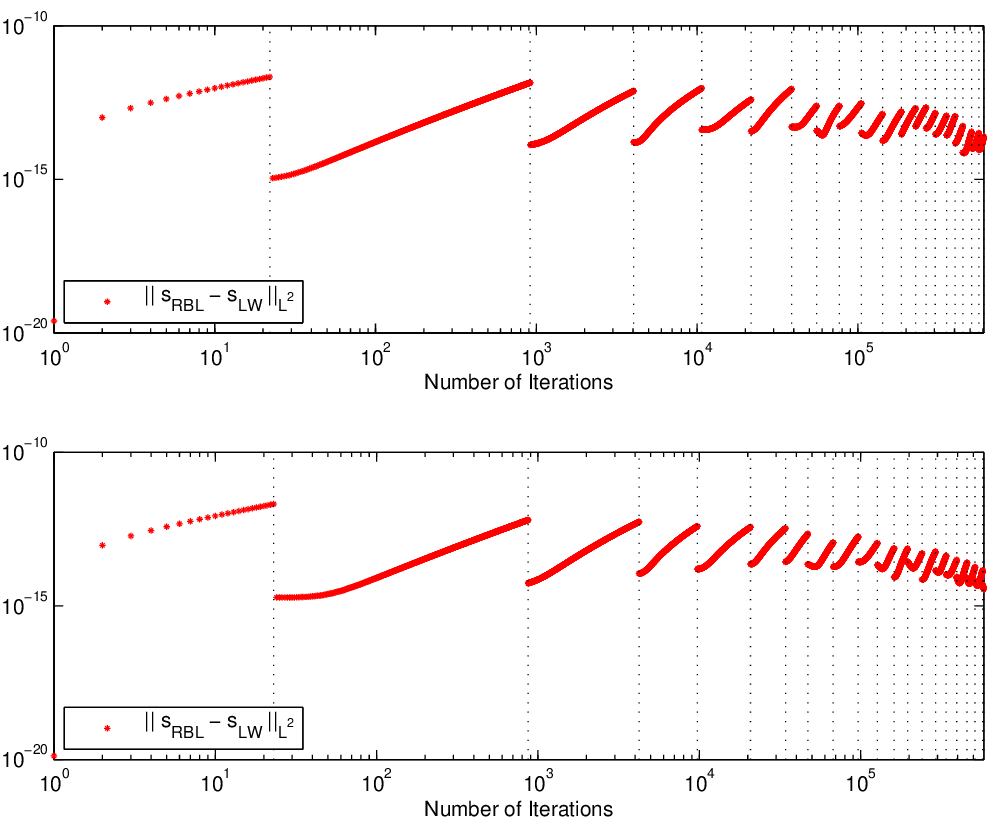}
\caption{Update error $\norm{s_{n,RBL} - s_{n,LW}}_{L^2(\Omega)}$ over the course of the iteration for the full setting (top) and the partial setting (bottom).}
\label{fig:RBL_vs_LW_Updates_RBLpaper}
\end{figure}

We observe the expected behaviour: the more inner iterations of the RBL method are performed for a given set of spaces, the worse the update error gets until one of the termination criteria is met. Note that in this test the alternative termination criterion always triggered except in the very end where the reduced discrepancy principle is reached. With $\norm{F_N(\bsigma_{RBL}) - F(\bsigma_{RBL})}_{L^2(\Omega)}\approx 1.296\cdot 10^{-8}$ and \req{eq:highDP_success_RBLpaper}, the high-dimensional discrepancy principle in line $2$ of Algorithm \ref{Algo:RBL_RBLpaper} is then met as well and the whole algorithm terminates. According to Remark \ref{remark:RBL_RBLpaper} this behaviour was expected.
 If we look at the iteration sequence in Figure \ref{fig:RBL_vs_LW_Updates_RBLpaper} we can observe two further aspects: the more outer iterations are performed, the better the set of reduced spaces fits the region of the parameter space relevant for the solution of the inverse problem, resulting in the update error to decrease as a whole. Finally, the space updates are performed more frequently in the beginning of the iteration sequence (note the logarithmic scale of the iteration axis) to quickly adapt to the region of interest. Once the spaces are well suited, more and more inner iterations per outer iteration can be performed, resulting in Algorithm \ref{Algo:RBL_RBLpaper} to outperform Algorithm \ref{Algo:LW_RBLpaper} by more than an order with respect to the computational time. Again, similar observations can be made in the partial setting as it can be seen on the bottom of Figure \ref{fig:RBL_vs_LW_Updates_RBLpaper}.

In Remark \ref{remark:notheory_for_RBL_RBLpaper} we justified the current lack of a theoretical investigation of the RBL method. Still, we can provide an experiment regarding its numerical regularization property. In the usual full setting the error $\norm{\sigma_{RBL} - \sigma^+}_{L^2(\Omega)}$ is shown over the decreasing relative noise level $\delta$ in Figure \ref{fig:RBL_relnoise_errorconv_RBLpaper}.

\begin{figure}[h]
\centering\includegraphics[scale = 0.5]{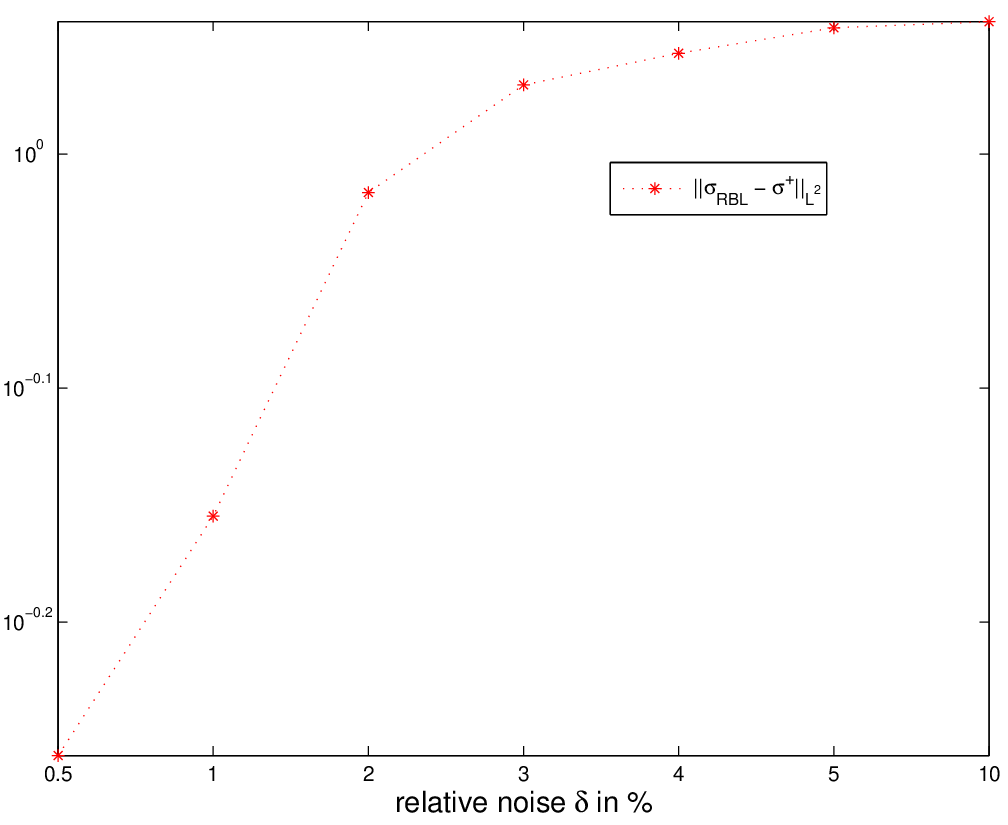}
\caption{Error $\norm{\sigma_{RBL} - \sigma^+}_{L^2(\Omega)}$ over the decreasing relative noise level $\delta$.}
\label{fig:RBL_relnoise_errorconv_RBLpaper}
\end{figure}

 By this we can conclude that a numerical regularization property is present for the RBL method.


\section{Conclusion}
\label{sec:Conclusion_RBLpaper}
For the problem of reconstructing the conductivity in the stationary heat equation it was investigated how reduced basis methods and the nonlinear Landweber method can be combined to reduce the overall computational time. A direct approach was shortly discussed to be inapplicable in the presence of high-dimensional parameter spaces. A new approach, the RBL method, was presented. It combined adaptive space enrichment of the reduced basis spaces with the nonlinear Landweber method. Using the RBL method, high-resolution images of the conductivity could be reconstructed, with the method being as accurate as the nonlinear Landweber method but roughly $13$ times faster.
Future work should contain a theoretical investigation of the RBL method including convergence theory as well as the application of its methodology to other inverse problems and more sophisticated regularization algorithms.

\ackn{We thank Vladimir Druskin for pointing out \cite{Druskin_Zaslavski} and introducing us to his approach.}

\section*{Appendix}
\appendix
\setcounter{section}{1}

We prove the statement made in \req{eq:fd_property_Linfty_RBLpaper} of Section \ref{sec:Problem formulation_RBLpaper}.

For $\sigma\in L_+^\infty(\Omega)$ and $\kappa\in L^2(\Omega)$ with $\sigma + \kappa \in L_+^\infty(\Omega)$ let $u^\sigma$, $u^{\sigma + \kappa}$ denote the corresponding solutions of \req{eq:forward_operator_RBLpaper}, $v_\kappa^\sigma$ the solution of \req{eq_frechetderivative_RBLpaper} and $\uc := \ei(\sigma)$. 
Note that $u^\sigma$ solving \req{eq:forward_operator_RBLpaper} is equivalent to $u^\sigma$ solving $b(u^\sigma ,v;\sigma) = f(v)$ for all $v\in H_0^1(\Omega)$. It holds for all $w\in H_0^1(\Omega)$
\begin{eqnarray}
&\int_\Omega (\sigma + \kappa)(\nabla u^{\sigma + \kappa} \cdot \nabla w) - (\sigma\nabla u^\sigma \cdot \nabla w)\,dx = 0 \nonumber\\
\label{App_proof_FD:eq1_RBLpaper}
\Leftrightarrow &\int_\Omega \sigma (\nabla (u^{\sigma + \kappa} - u^\sigma)\cdot\nabla w)\,dx = -\int_\Omega \kappa \nabla u^{\sigma + \kappa}\cdot\nabla w\,dx .
\end{eqnarray}
The test function $w = u^{\sigma + \kappa} - u^\sigma$ in \req{App_proof_FD:eq1_RBLpaper} and the H\"older inequality yield
\begin{eqnarray*}
\fl
&\uc\norm{\nabla(u^{\sigma + \kappa} - u^\sigma)}_{L^2(\Omega)}^2 \leq \int_{\Omega}\sigma (\nabla (u^{\sigma + \kappa} - u^\sigma)\cdot\nabla (u^{\sigma + \kappa} - u^\sigma)\,dx\\
\fl
& =  -\int_\Omega \kappa \nabla u^{\sigma + \kappa}\cdot\nabla (u^{\sigma + \kappa} - u^\sigma)\,dx
\leq \norm{\kappa}_\infty \norm{\nabla u^{\sigma + \kappa}}_{L^2(\Omega)}  \norm{\nabla (u^{\sigma + \kappa} -u^\sigma)}_{L^2(\Omega)}\\
\fl
&\Leftrightarrow \norm{\nabla(u^{\sigma + \kappa} - u^\sigma)}_{L^2(\Omega)} \leq \frac{\norm{\kappa}_\infty}{\uc}\norm{\nabla u^{\sigma + \kappa}}_{L^2(\Omega)}.
\end{eqnarray*}
Similar arguments for $\sigma + \kappa$ and test function $w = u^{\sigma + \kappa}$ in \req{eq:forward_operator_RBLpaper} together with the inequality of Poincar\'{e}-Friedrich and the notation $c :=\norm{1}_{L^2(\Omega)}$ yield
\begin{eqnarray*}
\fl
\uc\norm{\nabla u^{\sigma + \kappa}}_{L^2(\Omega)}^2 &\leq \int_{\Omega}(\sigma + \kappa ) (\nabla u^{\sigma + \kappa}\cdot\nabla u^{\sigma + \kappa} )\,dx = -\int_\Omega 1 \cdot u^{\sigma + \kappa}\,dx\leq  \norm{1}_{L^2(\Omega)}  \norm{u^{\sigma + \kappa}}_{L^2(\Omega)}\\
\fl
 &\leq c\cdot C_{PF} \norm{\nabla u^{\sigma + \kappa}}_{L^2(\Omega)}.
\end{eqnarray*}
Introducing the constant $C' := \frac{c\cdot C_{PF}}{\uc^2}$ we get
\begin{eqnarray}
\label{App_proof_FD:eq2_RBLpaper}
\norm{\nabla(u^{\sigma + \kappa} - u^\sigma)}_{L^2(\Omega)} \leq C' \norm{\kappa}_\infty .
\end{eqnarray}
With \req{App_proof_FD:eq2_RBLpaper} and the definition of $v_\kappa^\sigma$ in \req{eq_frechetderivative_RBLpaper} it follows
\begin{eqnarray}
\fl
 &\int_\Omega \sigma \nabla (u^{\sigma + \kappa} - u^\sigma)\cdot\nabla w - \sigma \nabla v_\kappa^\sigma \cdot \nabla w\,dx = -\int_\Omega \kappa \nabla u^{\sigma + \kappa}\cdot\nabla w\,dx + \int_\Omega \kappa \nabla u^\sigma\cdot\nabla w\,dx\nonumber \\
\fl
\label{App_proof_FD:eq3_RBLpaper}
\Leftrightarrow &\int_\Omega \sigma \nabla (u^{\sigma + \kappa} - u^\sigma - v_\kappa^\sigma)\cdot\nabla w\,dx = \int_\Omega \kappa \nabla (u^\sigma - u^{\sigma + \kappa})\cdot\nabla w\,dx .
\end{eqnarray}
The test function $w = u^{\sigma + \kappa} - u^\sigma - v_\kappa^\sigma$ in \req{App_proof_FD:eq3_RBLpaper}, the H\"older inequality and \req{App_proof_FD:eq2_RBLpaper} yield
\begin{eqnarray*}
\fl
\uc\norm{\nabla (u^{\sigma + \kappa} - u^\sigma - v_\kappa^\sigma)}_{L^2(\Omega)}^2 
&\leq \int_\Omega \sigma \nabla (u^{\sigma + \kappa} - u^\sigma - v_\kappa^\sigma)\cdot\nabla (u^{\sigma + \kappa} - u^\sigma - v_\kappa^\sigma)\,dx\\ 
\fl
&= \int_\Omega \kappa \nabla (u^\sigma - u^{\sigma + \kappa})\cdot\nabla (u^{\sigma + \kappa} - u^\sigma - v_\kappa^\sigma)\,dx\\
\fl
&\leq \norm{\kappa}_\infty \norm{\nabla (u^\sigma - u^{\sigma + \kappa})}_{L^2(\Omega)} \norm{\nabla (u^{\sigma + \kappa} - u^\sigma - v_\kappa^\sigma)}_{L^2(\Omega)}\\
\fl
\Leftrightarrow \norm{\nabla (u^{\sigma + \kappa} - u^\sigma - v_\kappa^\sigma)}_{L^2(\Omega)} &\leq \frac{1}{\uc}\norm{\kappa}_\infty \norm{\nabla (u^\sigma - u^{\sigma + \kappa})}_{L^2(\Omega)} \leq \frac{C'}{\uc} \norm{\kappa}_\infty^2 .
\end{eqnarray*}
Using the inequality of Poincar\'{e}-Friedrich again and introducing $C'' := \frac{C_{PF}\cdot C'}{\uc}$ we get
\begin{eqnarray*}
\fl
\norm{u^{\sigma + \kappa} - u^\sigma - v_\kappa^\sigma}_{L^2(\Omega)} \leq C_{PF} \norm{\nabla (u^{\sigma + \kappa} - u^\sigma - v_\kappa^\sigma)}_{L^2(\Omega)} \leq C''\norm{\kappa}_\infty^2 .
\end{eqnarray*}
Since $C''$ is independant of $\kappa$ the statement follows
\begin{eqnarray*}
\fl
\lim_{\norm{\kappa}_\infty \rightarrow 0}\frac{\norm{\mcF(\sigma + \kappa) - \mcF(\sigma) - \mcF'(\sigma)\kappa}_{L^2(\Omega)}}{\norm{\kappa}_\infty} & =  
\lim_{\norm{\kappa}_\infty \rightarrow 0} \frac{\norm{u^{\sigma + \kappa} - u^\sigma - v_\kappa^\sigma}_{L^2(\Omega)}}{\norm{\kappa}_\infty} \\
\fl
&\leq \lim_{\norm{\kappa}_\infty \rightarrow 0} \frac{C''\norm{\kappa}_\infty^2}{\norm{\kappa}_\infty} = 0 .
\end{eqnarray*}

\section*{References}
\bibliography{Generalbib}
\bibliographystyle{abbrv}

\end{document}